\documentclass[reqno]{amsart}
\usepackage[utf8]{inputenc}
\usepackage[a4paper]{geometry}
\usepackage{tikz}
\usetikzlibrary{patterns,decorations.markings,arrows.meta} 
\usepackage[english]{babel}
\usepackage[document]{ragged2e}
\usepackage{amsfonts}
\usepackage{amsthm}
\usepackage{amssymb, amsmath}

\usepackage{hyperref}

\usepackage{xcolor}
\usepackage{cleveref}
\usepackage{tikz}
\usetikzlibrary{calc}

\usepackage{enumitem}

\theoremstyle{plain}
\newtheorem{theo}{Theorem}[section]
\newtheorem{theoA}{Theorem}

\newtheorem{prop}[theo]{Proposition}
\newtheorem{lemma}[theo]{Lemma}

\newtheorem{coro}[theo]{Corollary}

\theoremstyle{definition}
\newtheorem{rem}[theo]{Remark}
\newtheorem{example}[theo]{Example}

\newcommand{\C}{\mathbb{C}}
\newcommand{\D}{\mathbb{D}}
\newcommand{\R}{\mathbb{R}}

\title[Localization of the poles of the best M\"obius approximations of $f$]{On the localization of the poles of the best M\"obius approximations of $f$}

\dedicatory{To the memory of Professor Christian Pommerenke}

\date{\today}

\author[H. Arbel\'aez]{Hugo Arbel\'aez}
\address{Hugo Arbel\'aez, Escuela de Matem\'aticas, Universidad Nacional de Colombia, Medell\'{\i}n, Colombia}
\email{hjarbela@unal.edu.co}

\author[M. Chuaqui]{Martin Chuaqui}
\address{Martin Chuaqui, Facultad de Matemáticas, Pontificia Universidad Católica de Chile, Casilla 306, Santiago 22, Chile}
\email{mchuaqui@mat.uc.cl}

\author[R. Hernández]{Rodrigo Hernández}
\address{Rodrigo Hernández, Facultad de Ingeniería y Ciencias, Universidad Adolfo Ibáñez, Av. Padre Hurtado 750, Viña del Mar, Chile}
\email{rodrigo.hernandez@uai.cl}

\author[W. Sierra]{Willy Sierra}
\address{Willy Sierra, Departamento de Matem\'aticas, Universidad del Cauca, Popay\'{a}n, Colombia}
\email{wsierra@unicauca.edu.co}



\thanks{The first author was supported by the Universidad Nacional de Colombia. The fourth author thanks the Universidad del Cauca for providing time for this work through research project VRI ID 6235.}

\begin{document}

\maketitle

\begin{abstract}
We study the localization of the poles of the best Möbius approximations for locally univalent functions in the unit disk. Sharp geometric bounds for the pole function are established in terms of Pommerenke’s linear invariant orders, refining classical criteria for convexity and concavity. The behavior of poles is further analyzed for starlike mappings, convex functions of order $\alpha$, Janowski functions, and Robertson’s class. For polygonal mappings, we describe the regions covered by the poles and obtain exact multiplicity results. We also derive new convexity conditions based on bounds of the Schwarzian derivative.

\end{abstract}

\section{Introduction and preliminaries}

This paper is motivated by \cite{Chuaqui2024} where geometric properties of a conformal mapping of the unit disk were found to be determined by the location of the poles of the family of its best Möbius approximations (BMAs). An unexpected result was that a mapping is convex iff all the poles lie outside the disk, while it is concave iff all poles lie inside the disk [\cite{Chuaqui2024}, Theorems 2.2, 2.3]. We will offer here a simple geometric proof. In our study, we investigate further connections between a conformal map and the family of such poles. In Theorems 3.1 and 3.4, the upper and lower order of a function as defined by Pommerenke are shown to give sharp estimates for the location of the poles, while stronger forms of convexity are shown to give refinements of the main results in \cite{Chuaqui2024}. Proposition 4.10, and Theorems 4.6 and 4.8, provide sharp estimates for the location of the poles for starlike mappings, and for the Janowski and Robertson subclasses of conformal mappings. In this same section of the paper, we give a fairly complete description of the regions covered by the poles when mapping the disk onto a general polygon or its complement. The final section is devoted to related results on convexity that bring in the Schwarzian derivative in connection with the values of the second coefficient $a_2$.

\subsection{Best Möbius Approximation (BMA)}

The best Möbius approximation (BMA) to a locally injective analytic function $f(z)$ at a point $\zeta$ is the unique Möbius transformation $M_f(z, \zeta)$ (in $z$) that agrees with $f(z)$ to second order at $\zeta$. Explicitly,
\begin{equation}
M_f(z, \zeta) = f(\zeta) + \frac{(z - \zeta)f'(\zeta)}{1 - \frac{1}{2}(z - \zeta)\dfrac{f''}{f'}(\zeta)}.
\end{equation}

Thus, the pole \( P_f(\zeta) \) of the BMA of \( f \) at \( \zeta \) is given by:
\[
P_f(\zeta) = \zeta + 2 \frac{f'}{f''}(\zeta),
\]
which includes the case \( P_f = \infty \) when \( f''(\zeta) = 0 \). In \cite{Chuaqui2024}, the authors gave a characterization of the localization of the poles of the BMAs as follows:

\begin{theoA} 
Let $f$ be locally injective in a domain $D$, and let $\zeta \in D$. Then the pole $P_f(\zeta)$ of the BMA of $f$ at $\zeta$ satisfies the following:
\begin{enumerate}[wide]
    \item [$(i)$] The modulus $|P_f(\zeta)|$ is bigger than, equal to, or smaller than $|\zeta|$ according to whether \textup{Re}$\{1 + \zeta (f''/f')(\zeta)\}$ is positive, zero, or negative.
    \item [$(ii)$]  The points $0, \, \zeta, \, P_f(\zeta)$ are collinear if and only if \textup{Im}$\{\zeta (f''/f')(\zeta)\} = 0$.
    \item [$(iii)$] $P_f(\zeta) = -\zeta$ if and only if $1 + \zeta (f''/f')(\zeta) = 0$.
\end{enumerate}
\end{theoA}

Moreover, an analytic proof was given that $f$ is convex if and only if the poles are outside $\D$, while $f$ is concave if and only if the poles lie inside $\D$. We have found the following geometric proof more intuitive. Let $\D_r$ denote the subdisk $|\zeta|<r$,  $C_r$ be the circle $|\zeta|=r<1$ and $z\in C_r$ be fixed. Because of the order of contact at $z$, the circles $M_f(z, C_r)$ and $f(C_r)$ are tangent at $f(z)$ and have the same curvature there, with a sign equal to that of \textup{Re}$\{1 + z(f''/f')(z)\}$. Since conformal mappings preserve orientation, we see that the set $M_f(z, \D_r)$ will correspond to the bounded disk in the complement of $M_f(z, C_r)$ when the curvature is positive, while it will be the unbounded disk in that complement when the curvature is negative. This means that $M_f(z,\zeta)$ has no poles
in $\overline{\D_r}$ for positive curvature, and otherwise for negative curvature. This is depicted in the figure below. In the intermediate stages, the set $M_f(z, C_r)$ must be a straight line. 
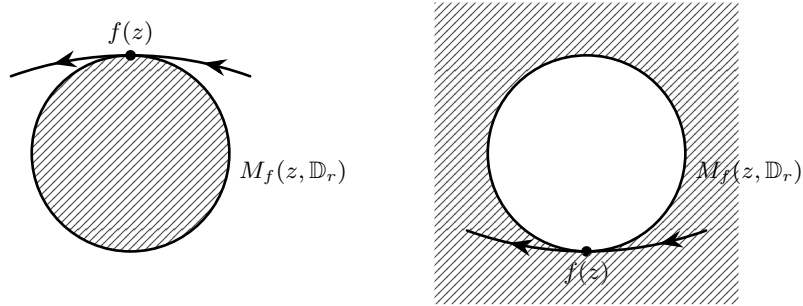
\begin{center}
    
\tikzset{
  flowarrows/.style={
    postaction={decorate},
    decoration={markings,
      mark=at position 0.20 with {\arrow[scale=1.15]{Stealth}}, 
      mark=at position 0.82 with {\arrow[scale=1.15]{Stealth}}  
    }
  }
}

\begin{tikzpicture}[line cap=round,line join=round,thick,every node/.style={font=\small}]

\begin{scope}[xshift=0cm]
  \def\R{1.3}       
  \def\rho{4.6}   

  \fill[pattern=north east lines,pattern color=black!70] (0,0) circle (\R);
  \draw[line width=1pt] (0,0) circle (\R);
  \node[right] at (\R*1,-0.25) {$M_f(z,\mathbb{D}_r)$};

  \path (90:\R) coordinate (F);
  \fill (F) circle (2pt);
  \node[above] at (F) {$f(z)$};

  \path (0,{\R-\rho}) coordinate (C);

  \draw[line width=1pt,flowarrows]
    (C) ++(70:\rho) arc[start angle=70, end angle=110, radius=\rho];
\end{scope}

\begin{scope}[xshift=6cm]
  \def\R{1.3}      
  \def\rho{4.6}    
  \def\Boundary{2} 

  \path (270:\R) coordinate (F);
  \fill (F) circle (2pt);
  \node[below] at (F) {$f(z)$};

  \path (0,{-\R+\rho}) coordinate (C);

  \fill[pattern=north east lines, pattern color=black!70, even odd rule]
      (-\Boundary,-\Boundary) rectangle (\Boundary,\Boundary)
      (0,0) circle (\R);

  \draw[line width=1pt] (0,0) circle (\R);
  \node[right] at (\R*1,-0.25) {$M_f(z,\mathbb{D}_r)$};

  \draw[line width=1pt,flowarrows]
    (C) ++(-70:\rho) arc[start angle=-70, end angle=-110, radius=\rho];
\end{scope}

\end{tikzpicture}
\end{center}

In the development of this article, the operator $A_f$ introduced by Pommerenke in \cite{Pomm1964} will play a fundamental role. If $f$ is a locally univalent analytic mapping defined in $\D$, we define
\[A_f(z)=\frac{(1-|z|^2)}{2}\frac{f''}{f'}(z)-\overline{z},\] for all $z\in \D$. From this operator, Pommerenke defined the (upper) order \cite{Pomm1964} and the lower order \cite{Pomm2007} by
\[\alpha_f=\sup_{z\in \mathbb{D}}|A_f(z)| \qquad \text{and} \qquad \mu_f=\inf_{z\in \mathbb{D}}|A_f(z)|.\]
One of the most important properties is that $\alpha_f=1$ if and only if $f$ is convex, and $\mu_f=1$ if and only if $f$ is concave (see \cite{Pomm1964} and \cite{Pomm2007}, respectively).

\section{Properties of $P_f$ function}
A straightforward computation shows the following result.
\begin{prop} \label{Propie}
Let $f$ be a locally univalent analytic mapping defined in $\D$, then
\begin{enumerate}[wide]
    \item [$(i)$] $P_{T \circ f}=P_f$, where $T(z)=az+b$, $a,b\in \C$, $a\neq 0$.
    \item [$(ii)$] If $\sigma_a(z)=(z+a)/(1+\overline{a}z)$, $a\in \D$, then
    \begin{equation}\label{prope 2}
     P_{f\circ \sigma_a}=\sigma_{-a}\circ P_f \circ \sigma_a.   
    \end{equation}
    \item [$(iii)$] Let $f$ be a locally univalent analytic mapping defined in $\D$ and $r\in (0,1)$. If $f_r(z):=f(rz)$, then 
    \begin{equation} \label{prope 3}
       P_{f_r}(z)=\frac1r\cdot P_f(rz), 
    \end{equation}
      for all $z\in \D$.
\end{enumerate}
\end{prop}

\begin{rem}
   If $f$ is a concave function, it follows from Theorem A and (\ref{prope 3}), that \[|P_{f_r}(z)|\leq \frac{1}{r}|rz|\leq 1.\]
   So, for $r\in (0,1)$, $f_r$ is a concave function. This tells us that concavity, in some sense, is a hereditary property. 
\end{rem}

In the following lemma, we show that the operators $A_f$ and $P_f$ are closely related; therefore, $A_f$ gives us information about the localizations of the poles of the BMA of $f$.

\begin{lemma} \label{lema P_f}
Let $f$ be a locally univalent analytic mapping defined in $\D$, then for any $z\in\D$ 
\[P_f(z)=\frac{w+z}{1+\overline{z}w},\quad\quad w=\frac{1}{A_f(z)}.\]
\end{lemma}

\begin{proof} Let $F=f\circ \sigma_a$, where $\sigma_a(z)=(z+a)/(1+\overline{a}z)$, $a\in \mathbb{D}$, is an automorphism of the unit disk. From Proposition \ref{Propie}
\[P_F=\sigma_{-a}\circ P_f \circ \sigma_a.
\]
On the other hand
\[P_F(z) = z + \displaystyle\frac{2(1 + \overline{a}z)^2}{\sigma_a'(z)(f''/f')(\sigma_a(z)) - 2\overline a(1 + \overline{a}z)}.\] 
Therefore, 
\begin{align}\label{A}
w=\frac{1}{A_f(a)} = \frac{2}{ (1-|a|^2)(f''/f')(a) - 2\overline{a}} = P_F(0)=\frac{P_f(a)-a}{1-\overline{a}P_f(a)}.
\end{align}

\end{proof}

\begin{rem}  An immediate consequence of the previous lemma is $|A_f(z)|>1$ if and only if $|P_f(z)|<1$, for all $z\in \D$. Thus, $f$ is a convex mapping if and only if $A_f(z)$ lies inside $\D$, which is equivalent to $w$ lies outside $\D$ and therefore also $P_f$. Analogously, $f$ is a concave mapping if and only if $A_f$ lies outside $\overline \D$ and therefore $P_f$ lies in $\D$. \\
  
\end{rem}

\section{Location of $P_f$ in terms of $\mu_f$ and $\alpha_f$}

\begin{theo}\label{order inf}
  Let $f$ be a locally univalent analytic mapping defined in $\D$. If $\mu_f>0$, then 
  \begin{equation} \label{E}
  |P_f(z)|\leq \frac{1+\mu_f|z|}{|z|+\mu_f}, 
  \end{equation}
for all $z\in\D$. The inequality (\ref{E}) is sharp.
\end{theo}

\begin{proof}
  Let $w=A_f^{-1}(z)$, for $z\in \D$. We have, from Lemma \ref{lema P_f}, that
\begin{align*}
    |P_f(z)|^2&= \left|\frac{w+z}{1+\overline{z}w}\right|^2=1-\frac{(1-|w|^2)(1-|z|^2)}{|1+\overline{z}w|^2}\\
    &\leq  1-\frac{\left(1-\dfrac{1}{\mu_f^2}\right)(1-|z|^2)}{\left(1+|z|\dfrac{1}{\mu_f}\right)^2}=\frac{(1+\mu_f|z|)^2}{(|z|+\mu_f)^2}.
\end{align*}
The example below shows the last statement.
\end{proof} 
We observe that Theorem \ref{order inf} also implies that if $f$ is a concave function, then $|P_f(z)|\leq 1$ for all $z \in \D$.
\begin{example}
    Let $a\in \R$, $a\neq 0$, and $f(z)=\left(\dfrac{1+z}{1-z}\right)^a$, $z\in \D$. Thus, $$P_f(z)=\dfrac{1+az}{a+z}=\dfrac{1}{\sigma_a(z)}.$$ We observe that $P_f(z)$ covers all $\C\setminus \D$ if $|a|<1$. 
It is known, see \cite{Pomm2008}, that 
\[\mu_f = \min\{|a|, 1\}\qquad \text{and} \qquad  \alpha_f = \max\{|a|, 1\}.\]
So, if $|a|<1$, it follows from Theorem \ref{order inf} that $$|P_f(z)|\leq \frac{1+|a||z|}{|a|+|z|},\quad\quad \forall z\in \D.$$

\end{example}

The following result generalizes the geometric observation obtained in the previous example.
\begin{prop}
For any $a\in \D$, $P_f(a)$ lies outside of $\sigma_a(D(0,1/\alpha_f))$. 
\end{prop}
\begin{proof}
Let $F=f\circ \sigma_a$, where $\sigma_a(z)=(z+a)/(1+\overline{a}z)$, $a\in \mathbb{D}$. It follows from \eqref{A} that $|P_F(0)|=\dfrac{1}{|A_f(a)|}\geq \dfrac{1}{\alpha_f}.$ Thus, from (\ref{prope 2}) we have $P_f(a)=\sigma_a \circ P_F(0).$
\end{proof}

The following result gives us another localization of $P_f$ in terms of the upper and lower orders of $f$.

\begin{theo} Let $f$ be a locally univalent analytic mapping defined in the unit disk. Then
\begin{enumerate}[wide]
   \item [$(i)$] \(\displaystyle\left|P_f(z)-\frac{\alpha_f^2-1}{\alpha_f^2-|z|^2}\cdot z\right|\geq \alpha_f\cdot\frac{1-|z|^2}{\alpha_f^2-|z|^2},\quad \)
for all $z\in \D$.
\item [$(ii)$]
\(\displaystyle\left|P_f(z)-\frac{\mu_f^2-1}{\mu_f^2-|z|^2}\cdot z\right|\leq \mu_f\cdot\frac{1-|z|^2}{\mu_f^2-|z|^2},\quad\) for all $|z|<\mu_f$.
\end{enumerate}
\end{theo}
\begin{proof}
 Let $z\in \D$ and $w=\dfrac{1}{A_f(z)}$. It is enough to observe, from Lemma \ref{lema P_f}, that
\begin{align*}
  \left|P_f(z)-\dfrac{\alpha_f^2-1}{\alpha_f^2-|z|^2}\cdot z\right|& =\left|\frac{w+z}{1+\overline{z}w}-\dfrac{\alpha_f^2-1}{\alpha_f^2-|z|^2}\cdot z\right|=\frac{1-|z|^2}{\alpha_f^2-|z|^2}\left|\frac{\alpha_f^2w+z}{1+\overline{z}w}\right|.
\end{align*}
Finally, using $\alpha_f|w|\geq 1$ for all $z\in \D$, a straightforward computation shows that $\left|\dfrac{\alpha_f^2w+z}{1+\overline{z}w}\right|\geq \alpha_f$, which proves $(i)$. To show $(ii)$, it suffices to follow the same calculations made in $(i)$ and use $\mu_f|w|\leq 1$.
\end{proof}

\begin{rem}
  It is easy to see that $\alpha_f=1$ in $(i)$ implies $|P_f(z)|\geq 1$, which is equivalent when $f$ is a convex mapping. In the same way, if $\mu_f=1$, which is the case when $f$ is a concave mapping, then from $(ii)$ it follows $|P_f(z)|\leq 1$ for all $z\in\D$. It also follows from $(i)$ that if $f\in S$, then
\[\left|P_f(z)-\frac{3z}{4-|z|^2}\right|\geq \frac{2(1-|z|^2)}{4-|z|^2},\]
for all $z\in \D$.  
\end{rem}

\section{Special Subclasses}
Since the location of $P_f$ for convex functions is already known, in this section we wish to study its location for other families of analytic functions. In particular, we investigate the case in which $f$ is convex of order $\alpha$.
\begin{theo} Let $f$ be a locally univalent analytic mapping defined in $\D$. If there is $t\geq 0$ such that $|P_f(z)|\geq 1+t$ for all $z\in\D$, then
\[\textup{Re} \left\{1+z\frac{f''}{f'}(z)\right\}\geq \frac{t}{2+t}.\]
\end{theo}

\begin{proof}
Since $f$ is a convex function, there is $h:\D \to \D$, analytic with $h(0)=0$, such that
\[1+z\frac{f''}{f'}(z)=\frac{1+h(z)}{1-h(z)}.\]
Now, if $h(z)=zg(z)$, it is easy to see that $P_f(z)=\displaystyle \frac{1}{g(z)}.$ It follows that $|h(z)|\leq \displaystyle \frac{1}{1+t}$ and so
\[\text{Re} \left\{\frac{1+h(z)}{1-h(z)}\right\}\geq \frac{t}{2+t},\]
for all $z\in\D$.
\end{proof}
The following example shows that the converse of the previous theorem is not true.

\begin{example}
Let $\alpha=\dfrac{t}{2+t}$ and $f$ a locally univalent analytic mapping defined in $\D$ such that $f(0)=f'(0)-1=0$ and satisfies
\begin{align} \label{B} 
    1+z\frac{f''}{f'}(z)=\frac{1+(1-2\alpha)z}{1-z}=:g_{\alpha}(z).
\end{align}
The function $g_{\alpha}$ maps $\D$ univalently onto the half-plane $\text{Re}\,w > \alpha$. So $f$ is convex of order $\alpha$ and also
\[P_f(z)=\frac{1-\alpha z}{1-\alpha}\quad \Rightarrow \quad  1\leq |P_f(z)|\leq \frac{1+\alpha}{1-\alpha}=1+t.\]
\end{example}

\begin{theo}
Let $f$ be a convex mapping of order $\alpha\in [0,1)$, defined in $\D$. Then 
\begin{align}\label{C}
 \left|P_f(z)+\frac{\alpha z}{1-\alpha}\right|\geq \frac{1}{1-\alpha},   
\end{align}
for all $z\in\D$. The inequality (\ref{C}) is sharp.
\end{theo}
\begin{proof}
    Since $f$ is a convex function of order $\alpha$, then
    \[\text{Re} \left\{1+z\frac{f''}{f'}(z)\right\}\geq \alpha,
    \]
    for all $z\in \D$. Thus, there is $h:\D \to \D$ with $h(0)=0$ such that 
    \[1+\frac{z}{1-\alpha}\frac{f''}{f'}(z)=\frac{1+h(z)}{1-h(z)},
    \]
    for all $z\in \D$.
    From where
    \[P_f(z)=z+2\frac{z}{1-\alpha}\frac{1-h(z)}{2h(z)}=\frac{z(1-\alpha h(z))}{(1-\alpha)h(z)}.\]
    It follows that
    \[\left|P_f(z)+\frac{\alpha z}{1-\alpha}\right|=\left|\frac{z}{h(z)}\frac{1}{1-\alpha}\right|\geq \frac{1}{1-\alpha},
    \]
    for all $z\in \D$. The last inequality is given by Schwarz's lemma.
For the function $f$ defined by (\ref{B}) the equality occurs in (\ref{C}) for all $z\in \D$. This proves the last statement.
\end{proof}

\begin{coro}If $f$ is a convex mapping of order $\alpha\in [0,1)$, then 
\begin{equation} \label{alpha1}
  |P_f(z)|\geq \frac{1-\alpha|z|}{1-\alpha},  
\end{equation}
for all $z\in \D.$ The inequality (\ref{alpha1}) is sharp.
\end{coro}

\begin{rem}
 We can assume 
 \[\text{Re} \left\{1+z\frac{f''}{f'}(z)\right\}\geq \alpha,\]
 $\alpha\in (-\infty,1)$, and with the same argument used to prove (\ref{C}) to obtain
 \begin{equation} \label{alpha2}
  |P_f(z)|\geq \frac{1-|\alpha||z|}{1-\alpha},  
\end{equation}
for all $z\in \D.$
In particular,
 \begin{enumerate}[wide]
\item [$(i)$] if $f\in \mathcal{F}_0(\lambda)$, $1/2\leq \lambda \leq 1,$ we have that   
 \begin{equation*}
  |P_f(z)|\geq \frac{2+(1-2\lambda)|z|}{1+2\lambda}, 
 \end{equation*}
for all $z\in \D.$ Namely, $f\in \mathcal{F}_0(\lambda)$, $1/2\leq \lambda \leq 1,$ if $f$ is a locally univalent analytic function, with $f(0)=f'(0)-1=0$, which satisfies  
\[\text{Re} \left\{1+z\frac{f''}{f'}(z)\right\}\geq \frac{1}{2}-\lambda,\]
for all $z \in \D$. $\mathcal{F}_0(\lambda)$, $1/2\leq \lambda \leq 1,$ is called the Ozaki class of close to convex functions (see \cite{Allu 2019}).
 
\item [$(ii)$] Umezawa [\cite{Umezawa1952}, Theorem 1] studied the functions $f$ defined in $\D$, analytics with $f(0)=f'(0)-1=0$, satisfying 
\begin{equation}\label{U}
  \text{Re} \left\{1+z\frac{f''}{f'}(z)\right\}\geq -\frac{\alpha}{2\alpha-3}, 
\end{equation}
where $\alpha > 3/2$, for all $z\in \D$. It follows by (\ref{alpha1}) and (\ref{U}) 
\[|P_f(z)|\geq \frac{(2-|z|)\alpha-3}{3(\alpha-1)},\]
for all $z\in \D$.
\end{enumerate} 
\end{rem}

Now, we consider the class of Janowski convex functions $f$ defined in $\D$ that satisfy the subordination relation
\begin{equation}\label{F}
1+z\frac{f''}{f'}(z)\prec \frac{1+Az}{1+Bz},
\end{equation}
where $-1\leq B < A\leq 1$. Note that the Möbius
transformation on the right side of (\ref{F}) maps $\D$ onto the disk (or half-plane) with diameter
$((1-A)/(1-B),(1+A)/(1+B))$. The class of all Janowski convex functions was first introduced
and studied by Janowski in \cite{Janos1973}.

\begin{theo} \label{Jano}
Let $f$ be a function in Janowski's class, then
\begin{equation} \label{G}
|P_f(z)|\geq \frac{2-|A+B||z|}{A-B},    
\end{equation}
for all $z\in \D$. The inequality (\ref{G}) is sharp.   
\end{theo}
\begin{proof}
 Since $f$ satisfies (\ref{F}) there is $h:\D \to \D$ analytic with $h(0)=0$ such that
 \[
 1+z\frac{f''}{f'}(z)=\frac{1+Ah(z)}{1+Bh(z)},
 \]
 for all $z\in \D$. So, 
\[\frac{f''}{f'}(z)=\frac{z(1+Bh(z))}{(A-B)h(z)}.
\]
It follows that
\[
P_f(z)=z\left[1+\frac{2(1+Bh(z))}{(A-B)h(z)}\right]=\frac{z}{h(z)}\frac{2+(A+B)h(z)}{A-B},
\]
then, from Schwarz's lemma
\[|P_f(z)|\geq \frac{2-|A+B||z|}{A-B},\]
for all $z\in \D$. To show that the inequality in (\ref{G}) is sharp, we consider the function 
$$f(z)= \left\{ \begin{array}{lcc} \dfrac{1}{A} \left[(1+Bz)^{A/B}-1\right], & if & A\neq 0, \, B\neq 0 \\ \\ \dfrac{1}{B}\log(1+Bz), & if & A=0 \\ \\ \dfrac{1}{A}(e^{Az}-1), & if & B=0, \end{array} \right.$$
for $z\in \D$. By direct calculation   
\[P_f(z)=\frac{2+(A+B)z}{A-B}.\]
\end{proof}

\begin{rem}
   We observe that if $A=1$ and $B=-1$ in (\ref{F}) then $f$ is convex, and by Theorem \ref{Jano} it follows $|P_f(z)|\geq 1$, for all $z\in \D$. Also, if $A=1$ and $B=-1+2\alpha$ then from (\ref{G}) we obtain (\ref{alpha1}).
\end{rem}

Similarly to what was studied for the Janowski class, we can consider the Robertson class, $S_{\alpha}$, introduced in \cite{Robert1969}. Let $f$ be a locally univalent function defined in $\D$ with $f(0)=0$ and $f'(0)=1$. We say $f\in S_{\alpha}$ if and only if
\[\text{Re}\left\{e^{i\alpha}\left(1+z\frac{f''}{f'}(z)\right)\right\}>0,
\]
for $|\alpha|<\pi/2$. In terms of subordination (see \cite{Firoz2023})
\begin{equation} \label{Rob}
 f\in S_{\alpha} \Longleftrightarrow  e^{i\alpha}\left(1+z\frac{f''}{f'}(z)\right) \prec \frac{ e^{i\alpha}+ e^{-i\alpha}z}{1-z}.
\end{equation}
Robertson also observed that for general values $\alpha$, $|\alpha|<\pi/2$, in $S_{\alpha}$, there are functions that are not univalent.

\begin{theo}
    Let $f\in S_{\alpha}$, $|\alpha|<\pi/2$, then 
\begin{equation}\label{Rob1}
|P_f(z)|\geq \frac{2-|1-e^{2i\alpha}||z|}{|1+e^{2i\alpha}|},   
\end{equation}
for all $z\in \D$. The inequality (\ref{Rob1}) is sharp.
\end{theo}

\begin{proof}
From (\ref{Rob}) there is $w:\D \to \D$ analytic with $w(0)=0$ such that
\[e^{i\alpha}\left(1+z\frac{f''}{f'}(z)\right) =\frac{ e^{i\alpha}+ e^{-i\alpha}w(z)}{1-w(z)}.\]
Thus
\[\frac{f''}{f'}(z)=\frac{1}{z}\left[\frac{(1+e^{-2i\alpha})w(z)}{1-w(z)}\right],\]
which implies
\begin{equation} \label{Rob2}
   P_f(z)=\frac{z}{w(z)}\left[\frac{2-(1-e^{-2i\alpha})w(z)}{1+e^{-2i\alpha}}\right]. 
\end{equation}
Inequality (\ref{Rob1}) follows from (\ref{Rob2}) and Schwarz's lemma. Now, we consider the function
\[f(z)=\frac{1}{\lambda}\left[(1-z)^{-\lambda}-1\right],\]
where $\lambda=e^{2i\alpha}$ and $|\alpha|<\pi/2.$ So $f\in S_{\alpha}$ and
\[\frac{f''}{f'}(z)=\frac{\lambda +1}{1-z} \Rightarrow P_f(z)= \frac{1}{\lambda+1}(2-(1-\lambda)z)=\frac{2-(1-e^{2i\alpha})z}{1+e^{2i\alpha}},\]
show that (\ref{Rob1}) is sharp.
\end{proof}

\begin{rem}
We note that $S_0$ is the class of convex functions and from (\ref{Rob1}), if $f\in S_0$, we obtain $|P_f(z)|\geq 1$, for all $z\in \D$.
\end{rem}

\begin{prop}\label{Starlike}
 Let $f:\D\to\C$ be a normalized starlike mapping. Then 
 \begin{equation}\label{Prstarlike}
  |P_f(z)-z|\geq \dfrac{1-|z|^2}{2+|z|},   
 \end{equation}
 for all $z\in\D.$ The equality occurs for some $z\in\D$ if and only if $f$ is a suitable rotation of the Koebe function.    
\end{prop}

\begin{proof}
 Since $f$ is a starlike and normalized function, we have
\begin{equation*}
\text{Re} \left\{z\frac{f'}{f}(z)\right\}>0,
\end{equation*}
for all $z\in \D$. So, there exists $w:\D \to \D$ an analytic function with $w(0)=0$ such that 
\begin{equation} \label{Star}
z\frac{f'}{f}(z)=\frac{1+w(z)}{1-w(z)},
\end{equation}
for all $z\in \D$. From (\ref{Star}) it follows that
\begin{equation*}
  \frac{f''}{f'}(z)=2\frac{zw'(z)+w(z)+w^2(z)}{z(1-w^2(z))} \Rightarrow P_f(z)-z=\frac{1-w^2(z)}{w'(z)+w_0(z)(1+w(z))},
\end{equation*}
where $w(z)=zw_0(z)$ with $w_0:\D \to \D$ analytic. Therefore, using the Schwarz Pick's lemma
\begin{align*}
    |P_f(z)-z|&\geq \frac{1-|w^2(z)|}{|w'(z)|+|w_0(z)|(1+|w(z)|)}\\
    &\geq\frac{(1-|z|^2)(1-|w(z)|^2)}{1-|w(z)|^2 +(1-|z|^2)|w_0(z)|(1+|w(z)|)}\\
    &\geq \frac{1-|z|^2}{1+|w_0(z)|(1+|w(z)|)}\geq \frac{1-|z|^2}{2+|z|}.
\end{align*}
Now, if the equality occurs in (\ref{Prstarlike}) for $z_0 \in \D$, then from the above chain of inequalities, it follows that
\[|w'(z_0)|(1-|z_0|^2)=1-|w(z_0)|^2.\]
Thus, Schwarz Pick's lemma implies that $w(z)=e^{i\theta}z$, for all $z\in \D$. Thus, from (\ref{Star}) we obtain $f(z)=\dfrac{z}{(1-e^{i\theta}z)^2}.$ 
\end{proof}

In general, the image of $P_f$, when $f$ is a starlike mapping, can be any domain of the complex plane, as shown in the following example.

\begin{example} Let $f$ be a conformal mapping defined in the unit disk onto a cross domain, given by
\begin{equation*}
  f'(z)=\frac{1}{(1-z)^4\sqrt{1+z^4}}.  
\end{equation*}
 A straightforward calculations shows that 
\begin{equation*}
  P_f(z)=\frac{1+z^4+2z^8}{z^3+3z^7}  
\end{equation*}
maps the unit disk onto $\C$.    
\end{example}

To generalize the result of the previous example, let us consider $f$ a conformal map of $\D$ onto a domain in the extended complex plane $\hat{\C}$, whose boundary consists of finitely many line segments, rays or lines. In \cite{Chuaqui2012}, it is shown that the pre-Schwarzian of $f$ has the form 
\begin{equation}\label{Poli}
   \frac{f''}{f'}(z)=\frac{2B_k(z)/B_m(z)}{1-zB_k(z)/B_m(z)}, 
\end{equation}
for some finite Blaschke products $B_k, B_m$ without common zeros. In \cite{Chuaqui2014} the authors showed that $k+1$ is equal to the number of convex
vertices, while $m$ is the number of concave vertices. It follows from (\ref{Poli}) that $P_f(z)=B_m(z)/B_k(z)$ for all $z\in \D$. Using this notation, we can prove the following theorem.

\begin{theo}
 Let $f$ be a conformal map of $\D$ onto a polygon in $\hat{\C}$.
 \begin{enumerate}[wide]
     \item [$(i)$] If $|b|>1$, then $\#\{P_f(z)=b\}=k$.
     \item [$(ii)$] If $|a|<1$, then $\#\{P_f(z)=a\}=m$.
 \end{enumerate}
 \end{theo}
 
\begin{proof}
   From Theorem A and the argument principle 
\begin{equation*}
    0 =\eta(\partial \D,b)=\eta(P_f(\partial \D),b) = \frac{1}{2\pi i}\int_{\partial \D} \frac{P'_f(z)dz}{P_f(z)-b}= \#\{P_f(z)=b\}-\#\{P_f(z)=\infty\}. 
\end{equation*}
Since $P_f(z)=B_m(z)/B_k(z)$, it follows that $\#\{P_f(z)=\infty\}=k$, and therefore $(i)$ follows from the above equation. Analogously, if $a\in \D$ 
\begin{equation*}
    m-k =\eta(P_f(\partial \D),a) = \#\{P_f(z)=a\}-\#\{P_f(z)=\infty\},
\end{equation*}
which proves $(ii)$.
\end{proof} 

\begin{rem}
Let $f$ map $\D $ onto the exterior of an $(n+2)$-gon, with the normalization $f(0) =\infty$. In \cite{Chuaqui2012} it was proven that for this mapping
\begin{equation*}
    z\frac{f''}{f'}(z)=\frac{2}{z^2B_k(z)/B_m(z)-1},
\end{equation*}
for some finite Blaschke products $B_k, B_m$ without common zeros. As before, in \cite{Chuaqui2014}, it was shown that $m$ and $k+2$ are equal to the number of concave and convex vertices of the polygon, respectively. If we observe that in this case $P_f(z)=z^3B_k(z)/B_m(z)$, for all $z\in \D$, we obtain the following result:
\end{rem}

\begin{theo}
 Let $f$ be a conformal map of $\D$ onto the exterior of a $(n+2)$-gon, with normalization $f(0) =\infty$.
 \begin{enumerate} [wide]
     \item [$(i)$] If $|b|>1$, then $\#\{P_f(z)=b\}=m$.
     \item [$(ii)$] If $|a|<1$, then $\#\{P_f(z)=a\}=k+3$.
 \end{enumerate}
 \end{theo}

\begin{prop} Let $f:\D\to\C$ be a starlike function of order $1/2$. Then 
\begin{equation}
|P_f(z)-z|\geq 1-|z|,\qquad z\in\D.
\end{equation}
The equality occurs for some $z\in\D$ if and only if $f$ is a suitable rotation of $f(z)=z/(1-z)$.
\end{prop}
\begin{proof}
    Since $f$ is starlike of order $1/2$, we have
\[\text{Re}\left\{2z\frac{f'}{f}(z)-1\right\}>0.\]
So, there is $w:\D \to \D$ with $w(0)=0$, such that
\begin{equation}\label{star1}
   2z\frac{f'}{f}(z)-1=\frac{1+w(z)}{1-w(z)}\Rightarrow \frac{f'}{f}(z)=\frac{1}{z(1-w(z))}. 
\end{equation}
By direct calculation, we obtain
\[\frac{f''}{f'}(z)=\frac{w(z)+zw'(z)}{z(1-w(z))}\Rightarrow P_f(z)-z=\frac{2z(1-w(z))}{w(z)+zw'(z)}.\]
Therefore, by Schwarz Pick's lemma
\begin{align}
 |P_f(z)-z|&\geq \frac{2(1-|w(z)|)}{\left|\dfrac{w(z)}{z}\right|+|w'(z)|}\geq \frac{2(1-|w(z)|)}{1+\dfrac{1-|w(z)|^2}{1-|z|^2}}\\
 & \geq \frac{(1-|w(z)|)(1-|z|^2)}{1-|w(z)|^2}\geq 1-|z|.
\end{align}
The proof of the last statement is analogous to the proof in Proposition \ref{Starlike}. We obtain $w=e^{i\theta}z$, for all $z\in \D$. And replacing in (\ref{star1}), it follows that $f(z)=z/(1-e^{i\theta}z).$
\end{proof}

Our next result addresses the class studied in the well known Noshiro-Warschawski theorem. 

\begin{prop} Let $f:\D\to\C$ be a normalized mapping such that {\rm Re}$\{f'(z)\}>0$. Then 
\begin{equation}\label{N}
|P_f(z)-z|\geq 1-|z|^2,\qquad z\in\D.
\end{equation}
The equality occurs for some $z\in\D$ if and only if $f$ is a suitable rotation of $f(z)=-2\log(1-z)-z$.
\end{prop}

\begin{proof}
 Since $f'(0)=1$ and \text{Re}$\{f'(z)\}>0$, there is $w:\D \to \D$ with $w(0)=0$, such that
\begin{equation}\label{N2}
  f'(z)=\frac{1+w(z)}{1-w(z)}.  
\end{equation}
So,
\[\frac{f''}{f'}(z)=\frac{2w'(z)}{1-w^2(z)} \Rightarrow P_f(z)-z=\frac{1-w^2(z)}{w'(z)}.\]
Using Schwarz Pick's lemma
\begin{equation}\label{N1}
 |P_f(z)-z|\geq\frac{1-|w^2(z)|}{|w'(z)|}\geq 1-|z|^2.  
\end{equation}
The last statement is obtained by assuming equality at a point in (\ref{N1}) and again using Schwarz Pick's lemma to conclude that $w(z)=e^{i\theta}z$, for all $z\in \D$. So, from (\ref{N2}) we have $f(z)=-2e^{-i\theta}\log(1-e^{i\theta}z)-z.$
\end{proof}

\section{$P_f$ and Schwarzian derivative}

Let $f$ be a locally univalent mapping, the Schwarzian derivative of $f$ is given by
\[Sf=\left(\dfrac{f''}{f'}\right)'-\frac12\left(\dfrac{f''}{f'}\right)^2.\] This expression captures conformal invariance, meaning it remains invariant under Möbius transformations. In the context of complex analysis, the Schwarzian derivative is used to measure how far a function deviates from being a Möbius transformation, providing a precise quantification of this deviation. This tool has found significant applications in Teichmüller theory, where quasiconformal deformations of Riemannian structures are studied, and in the theory of differential equations, where the solutions of certain equations can be characterized by the Schwarzian.\\
Let $N$ be the set of all analytic functions defined in $\D$ satisfying
\[|Sf(z)|\leq \frac{2}{(1-|z|^2)^2}.\]

The class $N_0$ is defined as the set of functions $f$ in $N$ normalized by $f(0)=f'(0)-1=f''(0)=0.$ These Nehari classes were formally introduced and extensively studied in \cite{Chuaqui1996}. A straightforward adaptation in the proof of Lemma 1 in \cite{Chuaqui1994} can be used to show the following result:

\begin{lemma} \label{Lemma f, F}
Let $f,\ F \in N_0$, $F(x)\in \R$ for $x\in (-1,1)$, such that $|Sf(z)|\leq SF(|z|)$, then $$\left|\dfrac{f''}{f'}(z)\right|\leq \dfrac{F''}{F'}(|z|).$$    
Equality holds for any $z\neq 0$ if and only if $f=F$.
\end{lemma}

We do not know of a result in which a bound on the Schwarzian norm of $f$ implies its convexity. In this direction, we prove the following theorem.

\begin{theo}\label{Teoconv}
Let $f$ be a locally univalent analytic mapping defined in $\D$, such that $f''(0)=0$.
\begin{enumerate}[wide]
    \item [$(i)$] If $|Sf(z)|\leq 2a^2$, where $a=0.653\ldots$ is the first positive solution of $2a\tan(a)=1$, then $f$ is a convex mapping.
  \item [$(ii)$] If $|Sf(z)|\leq n|z|^{n-1}-\dfrac{1}{2}|z|^{2n}$ for some $n\in \mathbb{N}$, then $f$ is a convex mapping. 
\end{enumerate}
\end{theo}

\begin{proof} We observe that
\begin{equation}\label{convex1}
 |P_f(z)| =\left|z+ \frac{2}{\dfrac{f''}{f'}(z)}\right| \geq \left|\frac{2}{\dfrac{f''}{f'}(z)}\right|-|z|.   
\end{equation}
Using the above lemma
\[ \left| \frac{f''}{f'}(z) \right| \leq 2a \tan(a|z|).\]
Therefore, from (\ref{convex1})
\[|P_f(z)| \geq \frac{1}{a \tan(a|z|)} - |z| \geq 1\Leftrightarrow (1+|z|) a \tan(a|z|) \leq 1.\]
Since $(1+|z|) a \tan(a|z|)$ is an increasing function of $|z|$, then in the first positive solution of $2a \tan(a) \leq 1$ we have
$|P_f(z)| \geq 1$ for all $z \in \mathbb{D}$, which implies that $f$ is convex and proves $(i)$. To show $(ii)$, we observe that if $S_F(z)=nz^{n-1}-\frac{1}{2}z^{2n}$ for some $n\in \mathbb{N}$, it follows from Lemma \ref{Lemma f, F} 
\[ \left| \frac{f''}{f'}(z) \right| \leq |z|^n.\]
So, from (\ref{convex1}) 
\[|P_f(z)|\geq \frac{2}{|z|^n}-|z|\geq 1,\]
for all $z\in \D$.
\end{proof}

\begin{rem}
It is important to note that the condition $f''(0)=0$ in the Theorem \ref{Teoconv} part $(i)$ cannot be omitted. For example, let $f(z)=e^{bz}$ with $b=2a$ and $z\in\D$. In this case $f''/f'(z)=2a$, then $f$ is not convex and $|S_f(z)|=2a^2$ for all $z\in \D$.
\end{rem}

\begin{coro}
  Let $f$ be a locally univalent analytic mapping defined in $\D$, such that $f''(0)=0$. If $|Sf(z)|\leq (m+1/2)|z|^{m}$ for some $m\in \mathbb{N}$, then $f$ is a convex mapping.
\end{coro}

\begin{proof}
We observe that if $x\in[0,1)$ 
\[nx^{n-1}-\frac{1}{2}x^{2n}=x^{n-1}(n-\frac{1}{2}x^{n+1})\geq (n-\frac{1}{2})x^{n-1}.\]
So, with $m=n-1$ we have
\[|Sf(z)|\leq (m+1/2)|z|^{m}\leq (m+1)|z|^m-\frac{1}{2}|z|^{2(m+1)},\]
for all $z\in \D$. It follows from Theorem \ref{Teoconv}, part $(ii)$, that $f$ is convex.
\end{proof}

\begin{prop}
  Let $f$ be a locally univalent analytic mapping defined in $\D$, such that $(1-|z|^2)^2|Sf(z)|\leq 2t$, $0<t\leq 1$, with $f''(0)=0$, then $P_f(z)$ lies outside the unit disk when $|z|\geq 1/(1+t)$. 
\end{prop}    

\begin{proof} Analogously, as before, we use the Lemma \ref{Lemma f, F} to prove that \[ \left| \frac{f''}{f'}(z) \right| \leq \frac{2t|z|}{1-|z|^2}, \] for all $z\in \D$. It follows that
\begin{align} \label{D}
|P_f(z)|\geq \left|\frac{2}{\dfrac{f''}{f'}(z)}\right|-|z|\geq \frac{1-|z|^2}{t|z|}-|z|. 
\end{align}
It is obvious that the expression on the right side of (\ref{D}) is greater than or equal to $1$ if $|z|\geq 1/(1+t)$.
\end{proof}

\begin{prop}
  Let $f\in N_0$ be such that $|P_f(z)|\leq k$, $k>1$, then $|z|\geq \dfrac{2}{\sqrt{k^2+8}+k}.$
\end{prop}
\begin{proof}
  Since $f\in N_0$, we have
\[\left| \frac{f''}{f'}(z) \right| \leq \frac{2|z|}{1-|z|^2},\]
 for all $z\in \D$. On the other hand, 
  \[P_f(z)\frac{f''}{f'}(z)=z\frac{f''}{f'}(z)+2,\]
  which implies
  \[2\leq (|P_f(z)|+|z|)\left| \frac{f''}{f'}(z) \right|\leq \frac{2|z|}{1-|z|^2}(|P_f(z)|+|z|).\]
  From where $2|z|^2+|P_f(z)||z|-1\geq 0$. And this is true if
  \[|z|\geq \frac{\sqrt{|P_f(z)|^2+8}-|P_f(z)|}{4}=\frac{2}{\sqrt{|P_f(z)|^2+8}+|P_f(z)|}.\]
  So, if $|P_f(z_0)|\leq k$, for some $z_0\in \D$, then
\[|z_0|\geq \dfrac{2}{\sqrt{k^2+8}+k}.
  \]
\end{proof}

\end{document}